\documentclass[a4paper, 12pt]{article}
\usepackage{geometry}                % See geometry.pdf to learn the layout options. There are lots.
\usepackage{graphicx}
\usepackage[british]{babel}
\usepackage{amssymb,amsmath,scalerel}
\usepackage{amsthm}
\usepackage{thmtools, thm-restate}
\usepackage[plainpages = false]{hyperref}
\usepackage[capitalize]{cleveref}
\usepackage[utf8]{inputenc}
\usepackage{enumerate}
\usepackage{tikz-cd}
\usepackage[backend = biber, style = numeric, giveninits]{biblatex}
\usetikzlibrary{babel}

\DeclareMathOperator*{\Times}{\scalerel*{\times}{\textstyle\sum}}

\DeclareMathOperator{\Aut}{Aut}

\DeclareMathOperator{\End}{End}
\DeclareMathOperator{\Fix}{Fix}

\DeclareMathOperator{\Hom}{Hom}
\DeclareMathOperator{\Id}{Id}

\DeclareMathOperator{\im}{Im}

\DeclareMathOperator{\Span}{span}
\DeclareMathOperator{\SpecR}{Spec_R}

\newcommand*{\eg}{e.g.\ }

\newcommand*{\fgtf}{finitely generated torsion-free }

\newcommand*{\grpgen}[1]{\left\langle{#1}\right\rangle}
\newcommand*{\grppres}[2]{\langle #1 \mid #2 \rangle}

\newcommand*{\ie}{i.e.\ }

\newcommand*{\inv}[1]{{#1}^{-1}}
\newcommand*{\M}{\mathcal{M}}
\newcommand*{\malclos}[1]{{#1}^{\Q}}
\newcommand*{\middlebar}{\ \middle | \ }
\newcommand*{\n}{\mathfrak{n}}

\newcommand*{\N}{\mathbb{N}}
\newcommand*{\Q}{\mathbb{Q}}

\newcommand*{\Rinf}{R_{\infty}}

\newcommand*{\restr}[2]{\left.#1\right|_{#2}}
\newcommand*{\size}[1]{\left| #1 \right|}

\newcommand*{\Z}{\mathbb{Z}}

\renewcommand{\phi}{\varphi}

\let\originalleft\left
\let\originalright\right
\renewcommand{\left}{\mathopen{}\mathclose\bgroup\originalleft}
\renewcommand{\right}{\aftergroup\egroup\originalright}

\declaretheorem[style=definition, name = Definition, numberwithin=section]{defin}

\declaretheorem[name = Theorem, sibling=defin]{theorem}
\declaretheorem[name = Lemma, sibling=defin]{lemma}
\declaretheorem[name = Proposition, sibling=defin]{prop}
\declaretheorem[name = Corollary, sibling=defin]{cor}
\declaretheorem[style=remark, name = Remark, numbered = no]{remark}

\declaretheorem[name = Theorem, numbered = no]{theorem*}

\declaretheorem[name = Question, numbered = no]{quest*}

\numberwithin{equation}{section}

\crefname{prop}{Proposition}{Propositions}
\crefname{quest}{Question}{Questions}
\crefname{cor}{Corollary}{Corollaries}

\title{The Reidemeister spectrum of direct products of nilpotent groups}
\author{Pieter Senden\footnote{Researcher funded by FWO-fellowship fundamental research (file number 1112522N)}}
\date{}

\begin{document}
\maketitle
\begin{abstract}
	We investigate the Reidemeister spectrum of direct products of nilpotent groups. More specifically, we prove that the Reidemeister spectra of the individual factors yield complete information for the Reidemeister spectrum of the direct product if all groups are finitely generated torsion-free nilpotent and have a directly indecomposable rational Malcev completion. We show this by determining the complete automorphism group of the direct product.
\end{abstract}
\let\thefootnote\relax
\let\thefootnote\relax
\section{Introduction}
Let \(G\) be a group and \(\phi \in \End(G)\). We define an equivalence relation on \(G\) by stating that \(x, y \in G\) are \emph{\(\phi\)-conjugate} if \(x = zy\inv{\phi(z)}\) for some \(z \in G\). The number of equivalence classes is called the \emph{Reidemeister number} \(R(\phi)\) of \(\phi\), and the set \(\SpecR(G) = \{R(\psi) \mid \psi \in \Aut(G)\}\) is called the \emph{Reidemeister spectrum} of \(G\). Finally, we say that \(G\) has the \emph{\(\Rinf\)-property} if \(\SpecR(G) = \{\infty\}\).

This article has both a topological and an algebraic motivation. Algebraically, we want to investigate the following question, which the author first posed in \cite{Senden21}:
\begin{quest*}
		Let \(G\) and \(H\) be (non-isomorphic) groups and let \(n \geq 2\) be an integer. What can we say about \(\SpecR(G^{n})\) and \(\SpecR(G \times H)\) in terms of \(\SpecR(G)\) and \(\SpecR(H)\)?
\end{quest*}
In the aforementioned article, we obtained a complete answer in the case where \(G\) and \(H\) are directly indecomposable and centreless. In this article, we investigate a somewhat diametrically opposite situation, namely the situation where all groups involved are nilpotent. The Reidemeister spectrum of direct products of nilpotent groups has already been investigated in some specific cases. S.\ Tertooy \cite{Tertooy19} proved that the Reidemeister spectrum of a direct product of finitely many free nilpotent groups of finite rank can be completely expressed in terms of the Reidemeister spectra of the individual factors. Recently, K.\ Dekimpe and M.\ Lathouwers \cite{DekimpeLathouwers21} proved a similar result for direct products of \(2\)-step nilpotent groups associated to graphs. The algebraic motivation is to find and prove a generalisation of both of those results.
%The main objective of this article is to prove a generalisation of both of these results.

As for the topological motivation, twisted conjugacy and Reidemeister numbers have their origin in topological Nielsen fixed-point theory, where they are used to determine bounds on the number of fixed points of a continuous self-map. We briefly describe this here; for more details, we refer to reader to \eg \cite{Jiang83}. 

Let \(X\) be a compact manifold and \(f : X \to X\) a continuous self-map. One of the goals of topological fixed-point theory is to find appropriate lower bounds on the number of fixed points of \(f\), or even on the number of fixed points of maps homotopic to \(f\). Using the universal cover of \(X\), one can subdivide \(\Fix(f) = \{x \in X \mid f(x) = x\}\) into (possibly empty) sets called fixed-point classes. The number of classes is called the Reidemeister number \(R(f)\) of \(f\). This number is a homotopy invariant, which means that if \(g\) is homotopic to \(f\), then \(R(g) = R(f)\). If we let \(f_{*}: \pi_{1}(X) \to \pi_{1}(X)\) denote the induced homomorphism on the fundamental group \(\pi_{1}(X)\) of \(X\), it moreover holds that \(R(f) = R(f_{*})\), where the latter is the (algebraic) Reidemeister number defined above. Thus, they coincide, which also explains why they carry the same name.

There is a way to assign to each fixed-point class (empty or not) a non-negative integer, called the index of the class. A class with index \(0\) is called inessential, otherwise it is called essential. The number of essential fixed-point classes of a map \(f\) is called the Nielsen number \(N(f)\). Like the Reidemeister number, the Nielsen number is a homotopy invariant. One can show that \(N(f) \leq \size{\Fix(g)}\) for any map \(g\) homotopic to \(f\), and for large classes of manifolds (\eg all manifolds of dimension at least \(3\)), there exists a map \(g\) homotopic to \(f\) that achieves equality. Consequently, the Nielsen number provides us a tight lower bound on the number of fixed points of maps homotopic to \(f\).

However, the Nielsen number is generally hard to compute, and requires as an intermediate step the computation of \(R(f)\). In many situations, though, the value of \(R(f)\) gives us all the information needed to determine \(N(f)\). Let \(G\) be a connected, simply connected, nilpotent Lie group, \(N\) a discrete, cocompact subgroup of \(G\) and define \(X\) to be the quotient space \(N \backslash G\). Then \(X\) is called a (compact) nilmanifold, and \(\pi_{1}(X) = N\). Furthermore, the following holds for every continuous self-map \(f: X \to X\) (see \eg \cite{HeathKeppelmann97}):
\[
	N(f) = \begin{cases}
		0	&	\mbox{if } R(f) = \infty,	\\
		R(f)	&	\mbox{if } R(f) < \infty.
	\end{cases}
\]
Therefore, studying (algebraic) Reidemeister numbers on direct products of nilpotent groups yields important tools to determine Nielsen numbers on products of nilmanifolds.

The aim of this article is to prove that the Reidemeister spectrum of a direct product of nilpotent groups is completely determined by the spectra of its factors if all factors are \fgtf and have a directly indecomposable rational Malcev completion. To do so, we fully describe the automorphism group of such a direct product. Descriptions of the automorphism group of direct products are known for other classes of groups, see \eg \cite{Bidwell08,BidwellCurranMcCaughan06,Johnson83,Senden21}. Using this description of the automorphism group, we then determine the Reidemeister spectrum.

This paper is structured as follows. In Section 2, we provide all the necessary preliminary and more technical results, ordered by theme. In Section 3, we first determine the automorphism group of a direct product of nilpotent groups satisfying the aforementioned conditions, after which we proceed to determining the Reidemeister spectrum. Finally, in Section 4, we show that this result is indeed a generalisation of the results by K.\ Dekimpe and M.\ Lathouwers, and S.\ Tertooy.

\section{Preliminaries}
\subsection{Reidemeister numbers}
To describe the automorphism group of a direct product, we use a representation of endomorphisms as described by \eg F.\ Johnson in \cite[\S1]{Johnson83}, J.\ Bidwell, M.\ Curran and D.\ McCaughan in \cite[Theorem~1.1]{BidwellCurranMcCaughan06}, and J.\ Bidwell in \cite[Lemma~2.1]{Bidwell08}.

Let \(G_{1}, \ldots, G_{n}\) be groups. Define
\[
	\M = \left\{ \begin{pmatrix} \phi_{11} & \ldots & \phi_{1n} \\ \vdots & \ddots & \vdots \\ \phi_{n1} & \ldots &\phi_{nn} \end{pmatrix} \middlebar \begin{array}{ccc} \forall 1 \leq i, j \leq n: \phi_{ij} \in \Hom(G_{j}, G_{i}) \\ \forall 1 \leq i, k, l \leq n: k \ne l \implies [\im \phi_{ik}, \im \phi_{il}] = 1\end{array}\right\}
\]
and equip it with matrix multiplication. Here, the addition of two homomorphisms \(\phi, \psi \in \Hom(G_{j}, G_{i})\) with commuting images is defined as \((\phi + \psi)(g) := \phi(g)\psi(g)\), and the multiplication of \(\psi \in \Hom(G_{i}, G_{k})\) and \(\phi \in \Hom(G_{j}, G_{i})\) is defined as \(\psi \phi := \psi \circ \phi\). It is readily verified that this puts a monoid structure on \(\M\), where the diagonal matrix with the respective identity maps on the diagonal is the neutral element.

\begin{lemma}[{See \eg \cite[Lemma~2.1]{Bidwell08}, \cite[Theorem~1.1]{BidwellCurranMcCaughan06}, \cite[\S1]{Johnson83}, \cite[Lemma~2.1]{Senden21}}]	\label{lem:MatrixRepresentationEndomorphismMonoid}
	For \(G = \Times\limits_{i = 1}^{n}G_{i}\), we have that \(\End (G) \cong \M\) as monoids.
\end{lemma}
We identify an endomorphism of \(G\) with its matrix representation and write \(\phi = (\phi_{ij})_{ij}\), where \(\phi_{ij} = \pi_{i} \circ \phi \circ e_{j}\) with \(\pi_{i}: G \to G_{i}\) the canonical projection and \(e_{j}: G_{j} \to G\) the canonical inclusion. In a matrix, we write the identity map as \(1\) and the trivial homomorphism as \(0\).

\begin{lemma}[{See \eg \cite[Lemma~2.2]{Senden21}}]	\label{lem:automorphismOfDirectProductImpliesNormalImages}
	With the notations as above, let \(\phi \in \Aut(G)\). Then for all \(1 \leq i \leq n\), \(G_{i}\) is generated by \(\{\im \phi_{ij} \mid 1 \leq j \leq n\}\).
\end{lemma}
Regarding Reidemeister spectra, we have the following general inclusion for direct products of groups.
\begin{defin}
	For \(a \in \N_{0} \cup \{\infty\}\), we define the product \(a \cdot \infty\) to be equal to \(\infty\). Let \(A_{1}, \ldots, A_{n}\) be subsets of \(\N_{0} \cup \{\infty\}\). We then define
	\[
		A_{1} \cdot \ldots \cdot A_{n} := \prod_{i = 1}^{n} A_{i} := \{a_{1} \ldots a_{n} \mid \forall i \in \{1, \ldots, n\}: a_{i} \in A_{i}\}.
	\]
	If \(A_{1} = \ldots = A_{n} =: A\), we also write \(A^{(n)}\) for the \(n\)-fold product of \(A\) with itself.
\end{defin}

\begin{prop}[{See \eg \cite[Corollary~2.6]{Senden21}}]	\label{prop:SpecRDirectProductCharacteristicSubgroups}
	Let \(G_{1}, \ldots, G_{n}\) be groups and put \(G = \Times\limits_{i = 1}^{n} G_{i}\). Then
	\[
		\left\{R(\phi) \middlebar \phi \in \Times\limits_{i = 1}^{n} \Aut(G_{i})\right\} = \prod_{i = 1}^{n} \SpecR(G_{i}) \subseteq \SpecR(G).
	\]
\end{prop}
Given a group \(G\) and an integer \(n \geq 1\), the symmetric group \(S_{n}\) embeds in \(\Aut(G^{n})\) in the following way:
\[
	S_{n} \to \Aut G^{n}: \sigma \mapsto (P_{\inv{\sigma}}: G^{n} \to G^{n}: (g_{1}, \ldots, g_{n}) \mapsto (g_{\inv{\sigma}(1)}, \ldots, g_{\inv{\sigma}(n)})),
\]
where the matrix representation of \(P_{\inv{\sigma}}\) is given by
\[
	P_{\inv{\sigma}} = \begin{pmatrix}
		e_{\inv{\sigma}(1)}	\\
		\vdots		\\
		e_{\inv{\sigma}(n)}
	\end{pmatrix}.
\]
Here, \(e_{i}\) is a row with a \(1\) on the \(i\)th spot and zeroes elsewhere. This inclusion combined with the subgroup \(\Aut(G)^{n}\) yields us the subgroup \(\Aut(G) \wr S_{n}\) of \(\Aut(G^{n})\).

For \(n\)-fold direct products of a group with itself, we can obtain a bit more information about the Reidemeister spectrum than just the inclusion in \cref{prop:SpecRDirectProductCharacteristicSubgroups}.
\begin{prop}[{See \eg \cite[Corollary~2.10]{Senden21}}]	\label{prop:ReidemeisterNumbersOfAut(G)wrSn}
	Let \(G\) be a group and \(n \geq 1\) a natural number. Then
	\[
		\{R(\phi) \mid \phi \in \Aut(G) \wr S_{n}\} = \bigcup_{i = 1}^{n} \SpecR(G)^{(i)} \subseteq \SpecR(G^{n}).
	\]
\end{prop}
For nilpotent groups, finally, we also need the following relation among the Reidemeister numbers associated to the groups in a central extension.
\begin{prop}	\label{prop:ReidemeisterNumberEqualsProductOfCentreAndQuotientNumber}
	Let \(N\) be a \fgtf nilpotent group and let \(\phi \in \Aut(N)\). Let \(\phi_{Z}\) and \(\bar{\phi}\) denote the induced automorphisms on \(Z(N)\) and \(N / Z(N)\), respectively. Then \(R(\phi) = R(\phi_{Z})R(\bar{\phi})\).
\end{prop}
\begin{proof}
	This follows from the proof of the well-known product formula for Reidemeister numbers in \fgtf nilpotent groups (see \eg \cite[Lemma~2.7]{Romankov11}).
\end{proof}

\subsection{Nilpotent groups}
We collect here the technical results needed in the derivation of the automorphism group of a direct product of nilpotent groups.
\begin{lemma}	\label{lem:centralisersFiniteIndexSubgroupsFGTFNilpotentGroups}
	Let \(N\) be a \fgtf nilpotent group and let \(H\) be a finite index subgroup. Then \(C_{N}(H) = Z(N)\).
\end{lemma}
\begin{proof}
	Let \(g \in C_{N}(H)\) and let \(n \in N\). As \(H\) has finite index in \(N\), there is a \(k \in \N_{0}\) such that \(n^{k} \in H\). For that \(k\), it then holds that \(gn^{k}\inv{g} = n^{k}\). As \(N\) is torsion-free nilpotent, it follows that \(gn\inv{g} = n\). Since \(n\) was arbitrary, this shows that \(g\) lies in \(Z(N)\).
	
	The inclusion \(Z(N) \leq C_{N}(H)\) is immediate.
\end{proof}
We need two properties of the Hirsch length of a polycyclic group. Recall that a \emph{polycyclic group} is a group admitting a subnormal series with cyclic factors. The \emph{Hirsch length} \(h(G)\) of a polycyclic group \(G\) is then defined as the number of infinite cyclic factors in such a series. This number is independent of the series chosen.
\begin{lemma}[{See \eg \cite[Exercise~8]{Segal83}}]	\label{lem:HirschLengthProperties}
	Let \(G\) be a polycyclic group, \(H\) a subgroup and \(N\) a normal subgroup. Then
	\begin{enumerate}[1)]
		\item \(h(H) \leq h(G)\) with equality if and only if \([G : H] < \infty\);
		\item \(h(G) = h(N) + h(G / N)\).
	\end{enumerate}
\end{lemma}
\subsubsection{Group homomorphisms}
The following can be seen as a Fitting lemma for \fgtf nilpotent groups.
\begin{lemma}	\label{lem:DirectDecompositionDueToEndomorphismFGTFNilpotentGroup}
	Let \(N\) be a \fgtf nilpotent group. Let \(\phi: N \to N\) be an endomorphism. Then there exists a \(k \in \N\) such that \(\ker \phi^{k} = \ker \phi^{k + n}\) for all \(n \geq 0\). Moreover, for such \(k\), \(\ker \phi^{k} \cap \im \phi^{k} = 1\). In particular, if \(\im \phi^{k}\) is normal in \(N\), then \(\im \phi^{k} \ker \phi^{k} \cong \im \phi^{k} \times \ker \phi^{k}\).
\end{lemma}
%%% Werkt ook voor polycyclische groepen
\begin{proof}
	Put, for each \(i \in \N_{0}\), \(K_{i} := \ker \phi^{i}\). Then clearly \(K_{i} \leq K_{i + 1}\) for each \(i\). Therefore,
	\[
		K_{1} \leq K_{2} \leq K_{3} \leq \ldots
	\]
	is a non-decreasing chain of subgroups. Since \(N\) is finitely generated and nilpotent, this chain must stabilise at some index, say \(k\). Thus, for each \(n \geq 0\), we have \(K_{k} = K_{k + n}\).
	
	Next, suppose that \(x \in \ker \phi^{k} \cap \im \phi^{k}\). Write \(x = \phi^{k}(y)\). Then \(\phi^{2k}(y) = \phi^{k}(x) = 1\), which implies that \(y \in K_{2k} = K_{k}\). Therefore, \(x = \phi^{k}(y) = 1\).
\end{proof}
Given a group \(G\) and a subgroup \(H\) of \(G\), we define the \emph{isolator} of \(H\) to be the set
\[
	\sqrt[G]{H} = \{g \in G \mid \exists n \in \N_{0}: g^{n} \in H\}.
\]
In general, \(\sqrt[G]{H}\) is not a subgroup. However, if \(G\) is a \fgtf nilpotent group, then \(\sqrt[G]{H}\) is a subgroup for every \(H \leq G\) (see \eg \cite[p.~249]{Kurosh60}). Given groups \(G_{1}, \ldots, G_{n}\) with respective subgroups \(H_{1}, \ldots, H_{n}\), it follows from the definition that, with \(G = \Times\limits_{i = 1}^{n} G_{i}\) and \(H = \Times\limits_{i = 1}^{n} H_{i}\),
\begin{equation}	\label{eq:isolatorDirectProducts}
	\sqrt[G]{H} = \Times_{i = 1}^{n} \sqrt[G_{i}]{H_{i}}.
\end{equation}
\begin{lemma}	\label{lem:isolatorCommutatorInKernelToTorsionfreeAbelian}
	Let \(G\) and \(A\) be two groups with \(A\) torsion-free abelian and let \(\phi: G \to A\) be a homomorphism. Then \(\sqrt[G]{\gamma_{2}(G)} \leq \ker \phi\).
\end{lemma}
\begin{proof}
	Since \(A\) is abelian, \(\gamma_{2}(G) \leq \ker \phi\). Now, let \(g \in G\) and \(k \in \N_{0}\) be such that \(g^{k} \in \gamma_{2}(G)\). Then \(\phi(g^{k}) = 0\), hence \(k\phi(g) = 0\) (we write \(A\) additively), so by torsion-freeness of \(A\), \(\phi(g) = 0\). Therefore, \(g \in \ker \phi\), which proves the lemma.
\end{proof}
The following lemma is a special case of \cite[Theorem~1]{Wehrfritz10}.
\begin{lemma}	\label{lem:endomorphsimsInducingAutomorphismOnCenterIsAutomorphismFGTFNilpotent}
Let \(N\) be a \fgtf nilpotent group. Let \(\phi \in \End(N)\). If \(\phi\) restricts to an automorphism of \(Z(N)\), then \(\phi\) is an automorphism of \(N\).
\end{lemma}
We also need the so-called (short) five lemma for groups.
\begin{lemma}	\label{lem:shortFiveLemmaGroups}
	Let \(G_{1}\) and \(G_{2}\) be groups with respective normal subgroup \(N_{1}\) and \(N_{2}\). Let \(\phi: G_{1} \to G_{2}\) be a morphism such that \(\phi(N_{1}) \leq N_{2}\). Let \(\phi_{N_{1}}: N_{1} \to N_{2}\) and \(\bar{\phi}: G_{1} / N_{1} \to G_{2} / N_{2}\) denote the induced homomorphisms. Consider the following commuting diagram with exact rows:
	\[
		\begin{tikzcd}
			1 \ar[r]	&	N_{1} \ar[r] \ar{d}{\phi_{N_{1}}}	&	G_{1} \ar[r] \ar{d}{\phi}	&	G_{1} / N_{1} \ar[r] \ar{d}{\bar{\phi}}	&	1	\\
			1 \ar[r]	&	N_{2} \ar[r] 		&	G_{2} \ar[r]	&	G_{2} / N_{2} \ar[r]	&	1
		\end{tikzcd}
	\]
	Suppose that \(\phi_{N_{1}}\) and \(\bar{\phi}\) are isomorphisms. Then \(\phi\) is an isomorphism as well.
\end{lemma}

The following technical result is proven by J.\ Bidwell \cite[Lemma~2.4]{Bidwell08} in the context of direct products of finite groups. His proof, however, does not use the finiteness of the groups, so it works for arbitrary groups.
\begin{lemma}[{See \eg \cite[Lemma~2.4]{Bidwell08}}]	\label{lem:imagesSigmaijNormal}	%Beter label verzinnen
	Let \(G_{1}, \ldots, G_{k}\) be groups and put \(G = \Times\limits_{i = 1}^{k} G_{i}\). Let \(\phi = (\phi_{ij})_{ij} \in \Aut(G)\) and write \(\inv{\phi} = (\phi'_{ij})_{ij}\). For \(i, j \in \{1, \ldots, k\}\), put
	\[
		\sigma_{i, j} = \sum_{\substack{l = 1 \\ l \ne j}}^{k} \phi'_{i, l} \phi_{l, i} \in \End(G_{i})
	\]
	and
	\[
		\sigma'_{i, j} = \sum_{\substack{l = 1 \\ l \ne j}}^{k} \phi_{i, l} \phi'_{l, i} \in \End(G_{i}).
	\]
	Then for all \(i, j \in \{1, \ldots, k\}\) and \(t \geq 0\), the subgroups \(\im(\sigma_{i, j}^{t})\) and \(\im\left( (\sigma'_{i, j})^{t}\right)\) are normal in \(G_{i}\).
\end{lemma}

\subsubsection{Rational completion}

\begin{defin}
	Let \(N\) be a \fgtf nilpotent group. The \emph{rational Malcev completion}, or rational completion for short, of \(N\) is the unique (up to isomorphism) torsion-free nilpotent group \(\malclos{N}\) satisfying the following properties (see \eg \cite{BaumslagMillerOstheimer16}, \cite[Chapter~6]{Segal83}):
	\begin{enumerate}[1)]
		\item \(N\) embeds into \(\malclos{N}\);
		\item for all \(n \in N\) and \(k \in \N_{0}\), there exists a unique \(m \in \malclos{N}\) such that \(m^{k} = n\);
		\item for all \(n \in \malclos{N}\), there exists a \(k \in \N_{0}\) such that \(n^{k} \in N\).
	\end{enumerate}
\end{defin}
We say that a group is \emph{directly indecomposable} if \(G \cong H \times K\) implies \(H = 1\) or \(K = 1\). Following \cite{Baumslag75}, we call a \fgtf nilpotent group \(N\) \emph{rationally indecomposable} if \(\malclos{N}\) is directly indecomposable. Examples of rationally indecomposable groups include \fgtf nilpotent groups with cyclic centre (see \eg \cite[Lemma~2]{Baumslag75}).
\begin{prop}[{See \eg \cite[Proposition~5]{BaumslagMillerOstheimer16}}]	\label{prop:decomposablefgtfNilpotentGroupHasDecomposableMalcevCompletion}
	Let \(N\) be a \fgtf nilpotent group. Suppose that \(N = N_{1} \times N_{2}\) for some groups \(N_{1}\) and \(N_{2}\). Then \(\malclos{N} = \malclos{N}_{1} \times \malclos{N}_{2}\).
\end{prop}
\begin{prop}	\label{prop:DirectlyIndecomposableMalcevCompletionImpliesDirectlyIndecomposable}
	Let \(N\) be a \fgtf nilpotent group. Then \(N\) is rationally indecomposable if and only if every finite index subgroup of \(N\) is directly indecomposable.
	
	In particular, every finite index subgroup of a rationally indecomposable group is itself rationally indecomposable.
\end{prop}
\begin{proof}
	First, assume that \(N\) is rationally indecomposable. Let \(H \leq N\) be a subgroup of finite index. We argue that \(\malclos{N}\) is also the rational completion of \(H\). The result then follows from \cref{prop:decomposablefgtfNilpotentGroupHasDecomposableMalcevCompletion}.
	
	Clearly, \(H\) embeds into \(\malclos{N}\), as it is a subgroup of \(N\). Next, let \(h \in H\) and \(k \in \N_{0}\). Since \(H \leq N\) and \(\malclos{N}\) is the rational completion of \(N\), there is a unique \(n \in \malclos{N}\) such that \(h = n^{k}\).
	
	Finally, let \(n \in \malclos{N}\). Then there is a \(k \in \N_{0}\) such that \(n^{k} \in N\). Since \(H\) has finite index in \(N\), there is an \(m \in \N_{0}\) such that \((n^{k})^{m} \in H\). Therefore, \(n^{km} \in H\), which finishes the proof that \(\malclos{N}\) is the rational completion of \(H\). \newline
	
	Conversely, suppose that \(\malclos{N} = A \times B\) is a non-trivial decomposition. Define \(N_{A} = A \cap N\) and \(N_{B} = B \cap N\). We prove that neither \(N_{A}\) nor \(N_{B}\) is trivial, that \(H := \grpgen{N_{A} N_{B}}\) is (isomorphic to) \(N_{A} \times N_{B}\) and that \(H\) has finite index in \(N\).
	
	Let \(a \in A\) be non-trivial. As \(a \in \malclos{N}\), there is a \(k \geq 1\) such that \(a^{k} \in N\). Consequently, \(a^{k} \in N_{A}\). As \(\malclos{N}\) is torsion-free, \(a^{k} \ne 1\), which proves that \(N_{A}\) is non-trivial. A similar argument holds for \(N_{B}\). 
	
	By construction, \(N_{A} \cap N_{B} = N \cap A \cap B = 1\). Also, the groups \(N_{A}\) and \(N_{B}\) commute, as \([N_{A}, N_{B}] \leq [A, B] = 1\). Thus, \(H = N_{A} \times N_{B}\). We are left with proving that \(H\) has finite index in \(N\). Let \(n \in N\). Then \(n = ab\) for some \(a \in A\) and \(b \in B\). By the properties of the Malcev completion, there exists a \(k \geq 1\) such that \(a^{k}, b^{k} \in N\). Then \(a^{k} \in N_{A}\) and \(b^{k} \in N_{B}\). Since \(a\) and \(b\) commute, \((ab)^{k} = a^{k}b^{k} \in N_{A} \times N_{B} = H\). Thus, some power of \(n\) lies in \(H\). This holds for arbitrary \(n\), so by \cite[Lemma~2.8]{Baumslag71a}, \(H\) has finite index in \(N\).
	
	Finally, let \(N\) be rationally indecomposable and \(H\) a finite index subgroup. Let \(K\) be a finite index subgroup of \(H\). Since \([N : K] = [N : H][H : K]\), \(K\) has finite in \(N\) and is therefore directly indecomposable. As \(K\) was arbitrary, every finite index subgroup of \(H\) is directly indecomposable. It follows that \(H\) is rationally indecomposable.
	\end{proof}
\begin{remark}
	There exist \fgtf directly (but not rationally) indecomposable nilpotent groups with a subgroup of finite index that does split as a non-trivial direct product. We provide an example based on \cite[\S~5]{Baumslag75} and \cite[\S~5]{BaumslagMillerOstheimer16}. Let \(A\) be the free abelian group on \(a, b\) and \(c\) and let \(B\) be the group \(A \rtimes_{\theta} \grpgen{t}\), where
	\[
		\theta(t): A \to A: a \mapsto ab, b \mapsto bc, c \mapsto c.
	\]
	Then \(B\) is finitely generated, torsion-free and 3-step nilpotent. Next, let \(F\) be the infinite cyclic group on \(f\) and put \(K = B \times F\). Fix an integer \(p \geq 2\). We define the group \(G_{p}\) to be the subgroup of \(\malclos{K}\) generated by \(K\) and the unique \(p\)-th root of \(bf\), which we call \(s\). It is proven in \cite[Lemma~3]{Baumslag75} that \(G_{p}\) is directly indecomposable (in their notation, \(G_{p} = G(1, p)\)). However, \(K\) is by construction a direct product.
	
	We now argue that \(K\) has finite index in \(G_{p}\). It is readily verified that \([t, \inv{s}]^{p} = c\). For ease of notation, we write \(c^{1/p} = [t, \inv{s}]\). Since \(c\) commutes with everything in \(K\), so does \(c^{1/p}\), as we work in a torsion-free nilpotent group. Analogously, \(s\) commutes with \(a, b, c\) and \(f\). In particular, \(\grpgen{f, c^{1/p}} \leq Z(G_{p})\). Using these observations and the rewriting rules
%	\[
%		at = tab, bt = tbc, ct = tc, t \inv{s} = \inv{s}t c^{1/p},
%	\]
	coming from the definition of \(B\) and \(c^{1/p}\), we can rewrite any element in \(G_{p}\) in the form
	\[
		a^{\alpha}b^{\beta}f^{F}t^{T}s^{S}c^{C/p},
	\]
	where \(\alpha, \beta, F, T, S, C \in \Z\). Any right coset of \(K\) in \(G_{p}\) then has a representative of the form \(s^{S} c^{C/p}\). Moreover, since \(s^{p} = bf \in K\) and \(c \in K\), it is sufficient to take \(S, C \in \{0, \ldots, p - 1\}\) to obtain a full set of representatives of the right cosets of \(K\) in \(G_{p}\). Therefore, \([G_{p} : K] \leq p^{2}\).
\end{remark}

\begin{lemma}	\label{lem:NormalSubgroupsFactorGroupsRationalClosure}
	Let \(N\) be a \fgtf nilpotent group and let \(K \leq N\) be a normal subgroup.
	\begin{enumerate}[1)]
		\item The rational completion \(\malclos{K}\), seen as a subgroup of \(\malclos{N}\), is normal in \(\malclos{N}\).
		\item Suppose that \(N / K\) is torsion-free. Then \(\malclos{(N / K)} \cong \malclos{N} / \malclos{K}\)
	\end{enumerate}
\end{lemma}
\begin{proof}
	For the first item, we refer the reader to \cite[p.~254]{Kurosh60}.
	
	For the second item, since \(N / K\) is torsion-free and finitely generated, its rational completion exists. Consider the map \(\phi: N\to \malclos{N} / \malclos{K}: n \mapsto n\malclos{K}\), which is the composition of the inclusion \(N \hookrightarrow \malclos{N}\) and the projection \(\malclos{N} \to \malclos{N} / \malclos{K}\). As \(K \leq \ker \phi\), it induces a homomorphism \(\Phi: N / K \to \malclos{N} / \malclos{K}\). We prove that \(\Phi\) is injective, and that \(\malclos{N} / \malclos{K}\) satisfies the properties of the rational completion for \(\im \Phi\).
	
	So, suppose that \(\Phi(nK) = \malclos{K}\) for some \(n \in N\). Then \(n \in \malclos{K}\). Hence, there is an \(m \in \N_{0}\) such that \(n^{m} \in K\). This means that \(n^{m}K = K\) in \(N / K\). As the latter is torsion-free, we conclude that \(nK = K\), which means that \(n \in K\). Hence, \(\Phi\) is injective.
	
	Next, let \(n \malclos{K} \in \im \Phi\) for some \(n \in N\) and let \(m \in \N_{0}\). Then there is a unique \(h \in \malclos{N}\) such that \(h^{m} = n\). Consequently, \((h\malclos{K})^{m} = n \malclos{K}\), so the second property is satisfied.
	
	Finally, let \(n\malclos{K} \in \malclos{N} / \malclos{K}\) for some \(n \in \malclos{N}\). Then there is an \(m \in \N_{0}\) such that \(n^{m} \in N\). Then \((n \malclos{K})^{m} \in \im \Phi\), so the third property is satisfied.
\end{proof}
\begin{prop}	\label{prop:NoAbelianFactorRationalClosureImpliesCentreInIsolatorCommutator}
	Let \(N\) be \fgtf nilpotent group. Suppose that \(\malclos{N}\) does not have an abelian direct factor. Then \(Z(N) \leq \sqrt[N]{\gamma_{2}(N)}\).
\end{prop}
\begin{proof}
	We proceed by contraposition. Let \(\pi: N \to \frac{N}{\sqrt[N]{\gamma_{2}(N)}}\) be the natural projection and suppose that \(z \in Z(N)\) is such that \(\pi(z) \ne 1\). By \cref{lem:NormalSubgroupsFactorGroupsRationalClosure}, we can consider the induced projection map \(\malclos{\pi}: \malclos{N} \to \malclos{N} / \malclos{\sqrt[N]{\gamma_{2}(N)}}\). Since \(\pi(z) \ne 1\), also \(\malclos{\pi}(z) \ne 1\). Indeed, suppose \(\malclos{\pi}(z) = 1\). Then \(z \in \malclos{\sqrt[N]{\gamma_{2}(N)}}\). Hence, \(z^{k} \in \sqrt[N]{\gamma_{2}(N)}\) for some \(k \geq 1\). Then \(\pi(z)^{k} = 1\), which implies that \(\pi(z) = 1\), as \(\frac{N}{\sqrt[N]{\gamma_{2}(N)}}\) is torsion-free.
	
	Now, since \(\malclos{N} / \malclos{\sqrt[N]{\gamma_{2}(N)}}\) is isomorphic to \(\malclos{(N / \sqrt[N]{\gamma_{2}(N)})}\) by \cref{lem:NormalSubgroupsFactorGroupsRationalClosure}, we see that \(\malclos{N} / \malclos{\sqrt[N]{\gamma_{2}(N)}}\) is abelian. Hence, it is isomorphic to \(\Q^{k}\) for some \(k \geq 1\). So, \(\malclos{\pi}(z)\) is part of a \(\Q\)-basis \(\mathcal{B}\) of \(\malclos{N} / \malclos{\sqrt[N]{\gamma_{2}(N)}}\).
	
	Finally, consider the composition
	\[
		\malclos{N} \to \malclos{N} / \malclos{\sqrt[N]{\gamma_{2}(N)}} \to \grpgen{\malclos{\pi}(z)}_{\Q},
	\]
	where the last group is the \(\Q\)-linear span of \(\malclos{\pi}(z)\) and where the last map sends \(\mathcal{B} \setminus \{\malclos{\pi}(z)\}\) to \(0\). Let \(\theta\) be this composition map. We claim that \(\malclos{N} \cong \ker \theta \times \malclos{\grpgen{z}}\). Clearly, \(\ker \theta\) is normal in \(\malclos{N}\). Since \(Z(\malclos{N}) = \malclos{Z(N)}\) (see \eg \cite[p.~257]{Kurosh60}), the subgroup \(\malclos{\grpgen{z}}\) lies in \(Z(\malclos{N})\), so it is central and hence normal. Finally, \(\ker \theta \cap \malclos{\grpgen{z}} = 1\) by construction, so \(\malclos{N}\) is the (internal) direct product of \(\ker \theta\) and \(\malclos{\grpgen{z}}\). The latter is abelian, which concludes the proof of the contraposition of the statement.
%	\cong \malclos{N / \sqrt[N]{\gamma_{2}(N)}}\). As \(z \sqrt[N]{\gamma_{2}(N)} \ne \sqrt[N]{\gamma_{2}(N)}\) in \(N / \sqrt[N]{\gamma_{2}(N)}\), also \(z \malclos{\sqrt[N]{\gamma_{2}(N)}} \ne \malclos{\sqrt[N]{\gamma_{2}(N)}}\) in \(\malclos{N/\sqrt[N]{\gamma_{2}(N)}}\).
\end{proof}
%\section{Automorphism group of direct product of nilpotent groups}\todo{structuur, volgorde en bindtekst}
\section{Main results}
\subsection{Automorphism group}
The goal of this section is to prove the following.
\begin{theorem}	\label{theo:AutomorphismGroupDirectProductFGTFNilpotentGroupsDirectyIndecomposableRationalClosure}
	Let \(N_{1}, \ldots, N_{k}\) be \fgtf non-abelian nilpotent groups such that \(N_{i}\) is rationally indecomposable for each \(i \in \{1, \ldots, k\}\). Put \(N = \Times\limits_{i = 1}^{k} N_{i}\). Let \(\phi = (\phi_{ij})_{ij} \in \Aut(N)\). Then the following hold:
	\begin{enumerate}[1)]
		\item For each \(i \in \{1, \ldots, k\}\), there is a unique \(\sigma(i) \in \{1, \ldots, k\}\) such that \(\phi_{i \sigma(i)}\) is an isomorphism.
		\item For each \(j \in \{1, \ldots, k\}\), there is a unique \(i \in \{1, \ldots, k\}\) such that \(\phi_{ij}\) is an isomorphism.
		\item For each \(i \in \{1, \ldots, k\}\) and for each \(j \in \{1, \ldots, k\}\) different from \(\sigma(i)\), we have \(\im \phi_{ij} \leq Z(N_{i})\).
	\end{enumerate}
\end{theorem}
To prove this, we proceed by induction on the number of distinct Hirsch lengths among the \(N_{i}\). First, we prove a last technical lemma, which is a major tool in the proof of the theorem above.
\begin{lemma}	\label{lem:automorphismDirectProductfgtfNilpotentGroupsContainsInjectivePerRow}
	Let \(N_{1}, \ldots, N_{k}\) be \fgtf non-abelian nilpotent groups such that \(N_{i}\) is rationally indecomposable for each \(i \in \{1, \ldots, k\}\). Put \(N = \Times\limits_{i = 1}^{k} N_{i}\) and let \(\phi = (\phi_{ij})_{ij} \in \Aut(N)\). Then for each \(i \in \{1, \ldots, k\}\), there is a \(j \in \{1, \ldots, k\}\) such that \(\phi_{ij}: N_{j} \to N_{i}\) is injective.
	
	Moreover, if for a given \(i \in \{1, \ldots, k\}\) the obtained index \(j \in \{1, \ldots, k\}\) satisfies \(h(N_{j}) = h(N_{i})\), we have the following:
	\begin{enumerate}[1)]
		\item For all \(l \in \{1, \ldots, k\}\) with \(l \ne j\), \(\im \phi_{il} \leq Z(N_{i})\);
		\item The index \(l \in \{1, \ldots, k\}\) for which \(\phi_{il}\) is injective is unique (namely \(j\));
		\item The map \(\phi_{ij}\) is an isomorphism.
	\end{enumerate}
\end{lemma}
\begin{proof}
	Write \(\inv{\phi} = (\phi'_{ij})_{ij}\) and fix \(i \in \{1, \ldots, k\}\). As \(N_{i}\) is not abelian and \(N_{i}\) is generated by the commuting subgroups \(\im \phi_{i1}\) up to \(\im \phi_{ik}\), by \cref{lem:automorphismOfDirectProductImpliesNormalImages}, there is a \(j \in \{1, \ldots, k\}\) such that \(K := \im \phi_{ij}\) is non-abelian. Let \(x \in K\). Then
	\[
		\left(\sum_{l = 1}^{k} \phi_{il} \phi'_{li}\right)(x) = x,
	\]
	as this is just part of the identity \(\phi \circ \inv{\phi} = \Id_{N}\). We can rewrite this as \(\sigma'_{i, j}(x) (\phi_{ij} \circ \phi'_{ji})(x) = x\), where \(\sigma'_{i, j}\) is defined as in \cref{lem:imagesSigmaijNormal}. As \((\phi_{ij} \circ \phi'_{ji})(x) \in \im \phi_{ij} = K\), it follows that \(\sigma'_{i, j}(x) \in K\). Since \(\im \phi_{il}\) commutes with \(\im \phi_{ij}\) for \(l \ne j\), \(\im \sigma'_{i, j}\) also commutes with \(\im \phi_{ij}\). As \(\im \sigma'_{i, j}\) is contained in \(K\), we obtain the inclusion \(\im \sigma'_{i, j} \leq Z(K)\). Now, put \(\hat{K} := \im (\phi_{ij} \circ \phi'_{ji})\). Since \(x = \sigma'_{i, j}(x) (\phi_{ij} \circ \phi'_{ji})(x)\) for each \(x \in K\), we see that \(K = Z(K) \hat{K}\). Since \(K\) is non-abelian, \(\hat{K}\) must be non-abelian as well.
	
	Next, let \(x \in \ker \phi_{ij}\). As
	\[
		\left(\sum_{l = 1}^{k} \phi'_{jl}\phi_{lj}\right)(x) = x,
	\]
	we get \(\sigma_{j, i}(x) = x\). Applying \cref{lem:DirectDecompositionDueToEndomorphismFGTFNilpotentGroup} to \(\sigma_{j, i} \in \End(N_{j})\) and using \cref{lem:imagesSigmaijNormal}, we get a \(t \geq 1\) such that \(H := \im \sigma_{j, i}^{t} \ker \sigma_{j, i}^{t} \cong \im \sigma_{j, i}^{t} \times \ker \sigma_{j, i}^{t}\). Moreover, \(H\) has finite index in \(N_{j}\). Indeed, if we compute the Hirsch length of \(H\) using \cref{lem:HirschLengthProperties}, we see that
	\[
		h(H) = h(\ker \sigma_{j, i}^{t}) + h(\im \sigma_{j, i}^{t}) = h(N_{j}).
	\]
	So, \(H\) indeed has finite index in \(N_{j}\) by the same lemma. Since \(N_{j}\) is rationally indecomposable, \cref{prop:DirectlyIndecomposableMalcevCompletionImpliesDirectlyIndecomposable} implies that either \(\ker \sigma_{j, i}^{t} = 1\) or \(\im \sigma_{j, i}^{t} = 1\). Suppose that \(\ker \sigma_{j, i}^{t} = 1\). Then \(\sigma_{j, i}^{t}\) is injective, which implies that \(\sigma_{j, i}\) is injective as well. Then \(h(\im \sigma_{j, i}) = h(N_{j})\), thus \(\im \sigma_{j, i}\) has finite index in \(N_{j}\). Now, remark that \([\im \sigma_{j, i}, \im \phi'_{j, i}] = 1\), as \(\im \phi'_{j, i}\) commutes with \(\im \phi'_{j, l}\) if \(l \ne i\). Therefore, \(\im \phi'_{j, i} \leq C_{N_{j}}(\im \sigma_{j, i})\), which lies in \(Z(N_{j})\) due to \cref{lem:centralisersFiniteIndexSubgroupsFGTFNilpotentGroups} since \(\im \sigma_{j, i}\) has finite index in \(N_{j}\). On the other hand, recall that \(\hat{K} = \im (\phi_{i, j} \circ \phi'_{j, i})\) is non-abelian. Therefore, \(\im \phi'_{j, i}\) is non-abelian as well. This contradicts the previously proven inclusion \(\im \phi'_{j, i} \leq Z(N_{j})\). Consequently, the assumption that \(\ker \sigma_{j, i}^{t} = 1\) is false, which implies that \(\im \sigma_{j, i}^{t} = 1\).
	
	Going back to the equality \(\sigma_{j, i}(x) = x\), the fact that \(\im \sigma_{j, i}^{t} = 1\) implies that \(x = 1\). As we took \(x \in \ker \phi_{ij}\), this finally proves that \(\ker \phi_{ij} = 1\), which means that \(\phi_{ij}\) is injective. Since \(i\) was arbitrary, the result follows. \newline
	
	Next, fix \(i \in \{1, \ldots, k\}\) and let \(j \in \{1, \ldots, k\}\) be an index such that \(\phi_{ij}\) is injective. Suppose that \(h(N_{i}) = h(N_{j})\). \cref{lem:HirschLengthProperties} then implies that \(\im \phi_{ij}\) has finite index in \(N_{i}\). Let \(l \in \{1, \ldots, k\}\) be distinct from \(j\). From \cref{lem:centralisersFiniteIndexSubgroupsFGTFNilpotentGroups} and the commuting conditions on \((\phi_{ij})_{ij}\) we derive that \(\im \phi_{il} \leq C_{N_{i}}(\im \phi_{ij}) \leq Z(N_{i})\), from which the first item follows. Moreover, \(N_{l}\) is non-abelian, so \(\phi_{il}\) cannot be injective, which proves the second item.
	
	Finally, to prove that \(\phi_{ij}\) is an isomorphism, note that by \cref{lem:isolatorCommutatorInKernelToTorsionfreeAbelian,prop:NoAbelianFactorRationalClosureImpliesCentreInIsolatorCommutator}, the inclusions \(Z(N_{l}) \leq \sqrt[N_{l}]{\gamma_{2}(N_{l})} \leq \ker \phi_{il}\) hold for each \(l \in \{1, \ldots, k\}\) different from \(j\). Let \(\phi_{Z}\) denote the induced automorphism on \(Z(N)\). Then the previous inclusions yield that the \(i\)th row of \(\phi_{Z}\) consists of zero maps everywhere except for the \(j\)th column; there we have \(\restr{\phi_{ij}}{Z(N_{j})}\). As \(\phi_{Z}\) is an automorphism, this restricted map must map into \(Z(N_{i})\) and must be surjective. The map \(\phi_{ij}\) is injective, hence so is it restriction. We conclude that \(\restr{\phi_{ij}}{Z(N_{j})}\) is an isomorphism between \(Z(N_{j})\) and \(Z(N_{i})\). In particular, both groups have the same Hirsch length.
	
	Now, we consider \(\bar{\phi}\), the induced automorphism on \(N / Z(N)\). Since \(\im \phi_{il} \leq Z(N_{i})\) for \(l \ne j\), the \(i\)th row of \(\bar{\phi}\) also consists of zero maps everywhere except for the \(j\)th column; there we have \(\bar{\phi}_{ij}: N_{j} / Z(N_{j}) \to N_{i} / Z(N_{i})\). As \(\bar{\phi}\) is an automorphism, \(\bar{\phi}_{ij}\) must be surjective. Recall that we assume that \(h(N_{j}) = h(N_{i})\) and that we have already proven that \(h(Z(N_{i})) = h(Z(N_{j}))\). Therefore, by \cref{lem:HirschLengthProperties}, we find that \(h(N_{j} / Z(N_{j})) = h(N_{i} / Z(N_{i}))\). Applying the first isomorphism theorem to \(\bar{\phi}_{ij}\) and using \cref{lem:HirschLengthProperties}, we get
	\[
		h(\ker \bar{\phi}_{ij}) = h(N_{j}/Z(N_{j})) - h(\im \bar{\phi}_{ij}) = h(N_{j}/Z(N_{j})) - h(N_{i}/Z(N_{i})) = 0,
	\]
	since \(\bar{\phi}_{ij}\) is surjective. Consequently, \(\ker \bar{\phi}_{ij}\) is finite. Since \(N_{j} / Z(N_{j})\) is torsion-free, \(\ker \bar{\phi}_{ij}\) must be trivial, which implies that \(\bar{\phi}_{ij}\) is injective as well. We conclude that \(\bar{\phi}_{ij}\) is an isomorphism.
	
	Summarised, both \(\restr{\phi_{ij}}{Z(N_{j})}\) and \(\bar{\phi}_{ij}\) are isomorphisms. \cref{lem:shortFiveLemmaGroups} then implies that \(\phi_{ij}\) is an isomorphism too, which proves the last item.
\end{proof}

\begin{proof}[Proof of \cref{theo:AutomorphismGroupDirectProductFGTFNilpotentGroupsDirectyIndecomposableRationalClosure}]
	As mentioned earlier, we proceed by induction on the number of distinct Hirsch lengths, say \(m\), among the \(N_{i}\). We start with the case \(m = 1\). By \cref{lem:automorphismDirectProductfgtfNilpotentGroupsContainsInjectivePerRow}, we find for each \(i \in \{1, \ldots, k\}\) a unique \(\sigma(i) \in \{1, \ldots, k\}\) such that \(\phi_{i\sigma(i)}\) is an isomorphism, since all factors have the same Hirsch length, and such that for \(j \in \{1, \ldots, k\}\) with \(j \ne \sigma(i)\), we have that \(\im \phi_{ij} \leq Z(N_{i})\). Consequently, both the first and third item are proven. We only have to argue that each column of \((\phi_{ij})_{ij}\) contains an isomorphism. In other words, we have to prove that \(\sigma: \{1, \ldots, k\} \to \{1, \ldots, k\}\) is surjective.
	
	To that end, note that the induced automorphism \(\bar{\phi}\) on \(N / Z(N)\) only has one non-zero map on each row, namely \(\bar{\phi}_{i\sigma(i)}\). As \(\bar{\phi}\) is an automorphism, there cannot be a zero column, which implies that \(\sigma\) must indeed be surjective. This proves the case \(m = 1\).	
	
	Now, suppose the result holds for \(m\) and that we now have \(m + 1\) distinct Hirsch lengths. Suppose, after reordering, that \(N_{l + 1}\) up to \(N_{k}\) are all the groups that have the lowest Hirsch length among the \(N_{i}\). 
	For \(i \in \{l + 1, \ldots, k\}\), \cref{lem:automorphismDirectProductfgtfNilpotentGroupsContainsInjectivePerRow} implies there is a \(j \in \{1, \ldots, k\}\) such that \(\phi_{ij}: N_{j} \to N_{i}\) is injective. Then \(h(N_{j}) \leq h(N_{i})\), which proves that \(j \in \{l + 1, \ldots, k\}\). Therefore, \(h(N_{j}) = h(N_{i})\), so the moreover-part of \cref{lem:automorphismDirectProductfgtfNilpotentGroupsContainsInjectivePerRow} yields that \(j =: \sigma(i)\) is the only index in \(\{1, \ldots, k\}\) such that \(\phi_{ij}\) is injective, that \(\im \phi_{ip} \leq Z(N_{i})\) for all \(p \ne \sigma(i)\), and that \(\phi_{i \sigma(i)}\) is an isomorphism. Furthermore, \cref{lem:isolatorCommutatorInKernelToTorsionfreeAbelian,prop:NoAbelianFactorRationalClosureImpliesCentreInIsolatorCommutator} combined imply that \(\phi_{ip}(Z(N_{p})) = 1\) for \(p \ne \sigma(i)\).

	Note that the same conclusions hold for \(\inv{\phi} = (\phi'_{ij})_{ij}\), \ie for each \(i \in \{l + 1, \ldots, k\}\), there is a unique \(\tau(i) \in \{l + 1, \ldots, k\}\) such that \(\phi'_{i\tau(i)}\) is an isomorphism, that \(\im \phi'_{i \tau(i)}\) has finite index in \(N_{i}\), and that, for \(j \ne \tau(i)\), both \(\im \phi'_{ij} \leq Z(N_{i})\) and \(\phi'_{ij}(Z(N_{j})) = 1\) hold.
	
	At this point, we thus have maps \(\sigma, \tau: \{l + 1, \ldots, k\} \to \{l + 1, \ldots, k\}\). We argue that both are bijective. To do so, we look at the induced maps \(\bar{\phi}\), \(\inv{\bar{\phi}}\) on \(N / Z(N)\). For each \(i \in \{l + 1, \ldots, k\}\) the \(i\)th row of \(\bar{\phi}\) and \(\inv{\bar{\phi}}\) consists of zero maps everywhere except for the \(\sigma(i)\)th and \(\tau(i)\)th column, respectively. From the equality \(\Id_{N / Z(N)} = \bar{\phi} \circ \inv{\bar{\phi}}\) we derive that
	\[
		\Id_{N_{i}} = \bar{\phi}_{i\sigma(i)} \circ \bar{\phi}'_{\sigma(i)i}
	\]
	holds for all \(i \in \{l + 1, \ldots, k\}\). The only non-zero map on the \(\sigma(i)\)th row is \(\bar{\phi}'_{\sigma(i)\tau(\sigma(i))}\). Therefore, \(\tau\) is surjective. By swapping the roles of \(\phi\) and \(\inv{\phi}\), we derive that \(\sigma\) is surjective as well.

	Next, write \(M_{1} = \Times\limits_{i = 1}^{l} N_{i}\) and \(M_{2} = \Times\limits_{i = l + 1}^{k} N_{i}\). Then \(M_{1} \times M_{2} = N\). Write
	\[
		\phi = \begin{pmatrix} \alpha & \beta \\ \gamma & \delta \end{pmatrix},
	\]
	in the notation of an endomorphism on \(M_{1} \times M_{2}\). Now, in this notation, we have proven in particular that \(\gamma(M_{1}) \leq Z(M_{2})\). Combining this with \eqref{eq:isolatorDirectProducts}, \cref{prop:NoAbelianFactorRationalClosureImpliesCentreInIsolatorCommutator} applied to each \(N_{i}\) and \cref{lem:isolatorCommutatorInKernelToTorsionfreeAbelian}, we get that \(Z(M_{1}) \leq \sqrt[M_{1}]{\gamma_{2}(M_{1})} \leq \ker \gamma\). Hence, the restriction \(\phi_{Z}\) of \(\phi\) to \(Z(N)\) has matrix representation
	\[
		\phi_{Z} = \begin{pmatrix} \restr{\alpha}{Z(M_{1})} & \restr{\beta}{Z(M_{2})} \\ 0 & \restr{\delta}{Z(M_{2})} \end{pmatrix}.
	\]
	Note that \(\phi_{Z}\) is an automorphism of \(Z(N)\). Now, define the map
	\[
		\psi = \begin{pmatrix} \alpha & \beta \\ 0 & \delta \end{pmatrix}.
	\]
	Then \(\psi\) is an endomorphism of \(N\), as \(\im \alpha\) and \(\im \beta\) commute. Its restriction \(\psi_{Z}\) to \(Z(N)\) coincides with \(\phi_{Z}\), so \cref{lem:endomorphsimsInducingAutomorphismOnCenterIsAutomorphismFGTFNilpotent} implies that \(\psi \in \Aut(N)\). Therefore, \(\alpha\) is injective.
 
Consequently, \(\im \alpha\) has finite index in \(M_{1}\). Since \(\im \beta\) commutes with \(\im \alpha\), this implies that \(\im \beta \leq Z(M_{1})\) by \cref{lem:centralisersFiniteIndexSubgroupsFGTFNilpotentGroups}. By a similar argument as for \(\gamma\), we get that \(Z(M_{2}) \leq \ker \beta\). Putting everything together, we get
	\[
		\psi_{Z} = \begin{pmatrix} \restr{\alpha}{Z(M_{1})} & 0 \\ 0 & \restr{\delta}{Z(M_{2})} \end{pmatrix}.
	\]
	We deduce that \(\restr{\alpha}{Z(M_{1})}\) is an automorphism of \(Z(M_{1})\), so \(\alpha\) is an automorphism itself, again by \cref{lem:endomorphsimsInducingAutomorphismOnCenterIsAutomorphismFGTFNilpotent}.
	
	Finally, we are able to use the induction hypothesis on \(M_{1}\), since \(M_{1}\) is a direct product of \fgtf non-abelian nilpotent rationally indecomposable groups, where there are \(m\) different Hirsch lengths among them. More precisely, the induction hypothesis applied to \(\alpha\) yields the following:
	\begin{itemize}
		\item For each \(i \in \{1, \ldots, l\}\), there exists a unique \(\sigma(i) \in \{1, \ldots, l\}\) such that \(\phi_{i\sigma(i)}\) is an isomorphism;
		\item For each \(i \in \{1, \ldots, l\}\) and \(j \in \{1, \ldots, l\}\) with \(j \ne \sigma(i)\), the map \(\phi_{ij}\) satisfies \(\im \phi_{ij} \leq Z(N_{i})\). Note, however, that we can immediately extend this to \(j \in \{1, \ldots, k\}\) with \(j \ne \sigma(i)\), as \(\im \phi_{ij}\) commutes with \(\im \phi_{i\sigma(i)} = N_{i}\);

		\item The map \(\sigma: \{1, \ldots, l\} \to \{1, \ldots, l\}\) is bijective.
	\end{itemize}
	
	Therefore, combining this with what we proved earlier, we get a bijective map \(\sigma: \{1, \ldots, k\} \to \{1, \ldots, k\}\) such that for all \(i,  j \in \{1, \ldots, k\}\) with \(j \ne \sigma(i)\), the map \(\phi_{i \sigma(i)}\) is an isomorphism and \(\im \phi_{ij} \leq Z(N_{i})\). This finishes the proof.	
\end{proof}
\cref{theo:AutomorphismGroupDirectProductFGTFNilpotentGroupsDirectyIndecomposableRationalClosure} yields necessary conditions on the matrix representation of an automorphism of a direct product. We now prove the converse (under slightly weaker conditions on the \(N_{i}\)), namely that each matrix \((\phi_{ij})_{ij}\) of morphisms satisfying the necessary conditions yields an automorphism of \(N\).
\begin{prop}\label{prop:converseInclusionAutomorphismGroupDirectProductFGTFNilpotentGroupsDirectyIndecomposableRationalClosure}
Let \(N_{1}, \ldots, N_{k}\) be \fgtf non-abelian nilpotent groups such that, for each \(i \in \{1, \ldots, k\}\), \(\malclos{N_{i}}\) has no abelian direct factors. Put \(N = \Times\limits_{i = 1}^{k} N_{i}\). Let \(\sigma \in S_{k}\) be a permutation. Suppose that we have a matrix \((\phi_{ij})_{ij}\) of morphisms, with \(\phi_{ij}: N_{j} \to N_{i}\), satisfying the following two conditions:
\begin{enumerate}[1)]
	\item For each \(i \in \{1, \ldots, k\}\), the map \(\phi_{i \sigma(i)}\) is an isomorphism.
	\item For all \(i, j \in \{1, \ldots, k\}\) with \(j \ne \sigma(i)\), the map \(\phi_{ij}\) satisfies \(\im \phi_{ij} \leq Z(N_{i})\).
\end{enumerate}
Then \(\phi := (\phi_{ij})_{ij}\) defines an automorphism of \(N\) under the identification of \cref{lem:MatrixRepresentationEndomorphismMonoid}.
\end{prop}
\begin{proof}
	The commuting conditions for \((\phi_{ij})_{ij}\) to define an endomorphism \(\phi\) are clearly met. Let \(i \in \{1, \ldots, k\}\) and let \(\phi_{i}: N \to N_{i}\) be the composition of \(\phi\) with the projection \(N \to N_{i}\). Then
	\[
		\phi_{i}(Z(N)) \leq \grpgen{\phi_{i1}(Z(N_{1})), \ldots, \phi_{ik}(Z(N_{k}))}.
	\]
	The subgroup on the right lies in \(Z(N_{i})\), since \(\phi_{ij}(Z(N_{j})) \leq \im \phi_{ij} \leq Z(N_{i})\) if \(j \ne \sigma(i)\) by assumption, and \(\phi_{i\sigma(i)}(Z(N_{\sigma(i)})) = Z(N_{i})\) as \(\phi_{i \sigma(i)}\) is an isomorphism. As \(i\) was arbitrary, we can deduce that \(\phi(Z(N)) \leq Z(N)\).
	
	Now, let \(\phi_{Z} = (\psi_{ij})_{ij}\) be the induced endomorphism on \(Z(N)\). For all \(i, j \in \{1, \ldots, k\}\), we have \(\psi_{ij} = \restr{\phi_{ij}}{Z(N_{j})}\). For \(j \ne \sigma(i)\), the map \(\psi_{ij}\) is the zero map, as \(Z(N_{j}) \leq \sqrt[N_{j}]{\gamma_{2}(N_{j})} \leq \ker \phi_{ij}\) by \cref{prop:NoAbelianFactorRationalClosureImpliesCentreInIsolatorCommutator,lem:isolatorCommutatorInKernelToTorsionfreeAbelian}, respectively. Each map \(\psi_{i\sigma(i)}\) is an isomorphism, being the restriction of an isomorphism to the centre. Since \(\sigma\) is a bijection, \((\psi_{ij})_{ij}\) is a matrix containing exactly one isomorphism in each row and each column, and zero maps in all other entries. Therefore, \(\phi_{Z}\) is an automorphism of \(Z(N)\). \cref{lem:endomorphsimsInducingAutomorphismOnCenterIsAutomorphismFGTFNilpotent} then implies that \(\phi\) is an automorphism of \(N\).
\end{proof}

\subsection{Reidemeister spectrum}
\begin{prop}	\label{prop:ReidemeisterNumberDirectProductCentralMaps}
	Let \(N_{1}, \ldots, N_{k}\) be \fgtf non-abelian nilpotent groups. Put \(N = \Times\limits_{i = 1}^{k} N_{i}\) and let \(\phi = (\phi_{ij})_{ij} \in \Aut(N)\). Suppose that \(\phi\) satisfies the following properties:
	\begin{enumerate}[1)]
		\item For each \(i \in \{1, \ldots, k\}\), there is a unique \(\sigma(i) \in \{1, \ldots, k\}\) such that \(\phi_{i\sigma(i)}\) is an isomorphism;
		\item For each \(i \in \{1, \ldots, k\}\) and \(j \ne \sigma(i)\), the map \(\phi_{ij}: N_{j} \to N_{i}\) satisfies \(\im \phi_{ij} \leq Z(N_{i})\) and \(\phi_{ij}(Z(N_{j})) = 1\).
	\end{enumerate}
	Define \(\psi = (\psi_{ij})_{ij}\), where
	\[
		\psi_{ij} = \begin{cases}
			\phi_{i \sigma(i)}	&	\mbox{if } j = \sigma(i)	\\
			0				&	\mbox{otherwise}.
		\end{cases}
	\]
	Then \(\psi\) is an automorphism of \(N\) and \(R(\phi) = R(\psi)\).
\end{prop}
\begin{remark}
	Explicitly requiring the condition \(\im \phi_{ij} \leq Z(N_{i})\) for \(j \ne \sigma(i)\) is redundant, as it follows from the first property and the commuting conditions on \((\phi_{ij})_{ij}\). Indeed, since \(\phi_{i\sigma(i)}\) is an isomorphism, \(\im \phi_{i \sigma(i)} = N_{i}\). By the commuting conditions, we then have for \(j \ne \sigma(i)\) that
	\[
		1 = [\im \phi_{ij}, \im \phi_{i \sigma(i)}] = [\im \phi_{ij}, N_{i}].
	\]
	This implies that \(\im \phi_{ij} \leq Z(N_{i})\). We include it, however, for the sake of consistency with the other results.
\end{remark}
\begin{proof}
	First, note that \(\psi\) is a well-defined endomorphism of \(N\), since there is only one map per row in the matrix representation of \(\psi\). Secondly, let \(\phi_{Z}\) and \(\psi_{Z}\) be the restrictions of \(\phi\) and \(\psi\), respectively, to \(Z(N)\). Then they have the same matrix representations, by the assumptions on \(\phi\) and the construction of \(\psi\). By \cref{lem:endomorphsimsInducingAutomorphismOnCenterIsAutomorphismFGTFNilpotent}, \(\psi\) is then an automorphism of \(N\).
%	Note also that \(\sigma: \{1, \ldots, k\} \to \{1, \ldots, k\}\) is a bijection; if it were not, then \(\phi_{Z}\) would have a zero column, contradicting the fact that \(\phi_{Z}\) is an automorphism.
	Finally, let \(\bar{\phi}\) and \(\bar{\psi}\) be the induced automorphisms on \(N / Z(N)\). Their matrix representations coincide as well by the assumptions on \(\phi\) and the construction of \(\psi\). Hence, combining this with \cref{prop:ReidemeisterNumberEqualsProductOfCentreAndQuotientNumber}, we obtain
	\[
		R(\phi) = R(\phi_{Z})R(\bar{\phi}) = R(\psi_{Z})R(\bar{\psi}) = R(\psi). \qedhere
	\]
\end{proof}
\begin{cor}
	Let \(N_{1}, \ldots, N_{k}\) be non-isomorphic \fgtf non-abelian nilpotent groups and let \(r_{1}, \ldots, r_{k}\) be positive integers. Put \(N = \Times\limits_{i = 1}^{k} N_{i}^{r_{i}}\) and \(r = r_{1} + \ldots + r_{k}\). Suppose every \(\phi = (\phi_{ij})_{ij} \in \Aut(N)\) satisfies the following conditions:
	\begin{enumerate}[1)]
		\item For each \(i \in \{1, \ldots, r\}\), there is a unique \(\sigma(i) \in \{1, \ldots, r\}\) such that \(\phi_{i\sigma(i)}\) is an isomorphism (which is then automatically an automorphism);
		\item For each \(i \in \{1, \ldots, r\}\) and \(j \ne \sigma(i)\), the map \(\phi_{ij}: N_{j} \to N_{i}\) satisfies \(\im \phi_{ij} \leq Z(N_{i})\) and \(\phi_{ij}(Z(N_{j})) = 1\).
	\end{enumerate}
	Then
	\[
		\SpecR(N) = \prod_{i = 1}^{k} \left(\bigcup_{j = 1}^{r_{i}} \SpecR(N_{i})^{(j)}\right).
	\]
\end{cor}
\begin{proof}
	By \cref{prop:ReidemeisterNumberDirectProductCentralMaps}, we can obtain the complete Reidemeister spectrum of \(N\) by only looking at the automorphisms lying in
	\[
		\Times_{i = 1}^{k} \left(\Aut(N_{i}) \wr S_{r_{i}}\right).
	\]
	The result now follows by combining \cref{prop:SpecRDirectProductCharacteristicSubgroups,prop:ReidemeisterNumbersOfAut(G)wrSn}.
\end{proof}
Combining the previous corollary with \cref{theo:AutomorphismGroupDirectProductFGTFNilpotentGroupsDirectyIndecomposableRationalClosure}, we get the following theorem.
\begin{theorem}	\label{theo:SpecRDirectProductFGTFNilpotentWithAbelian}
	Let \(N_{1}, \ldots, N_{k}\) be non-isomorphic \fgtf non-abelian nilpotent groups and let \(r_{1}, \ldots, r_{k}\) be positive integers. Put \(N = \Times\limits_{i = 1}^{k} N_{i}^{r_{i}}\). Suppose that, for each \(i \in \{1, \ldots, k\}\), \(N_{i}\) is rationally indecomposable. Then
	\[
		\SpecR(N) = \prod_{i = 1}^{k} \left(\bigcup_{j = 1}^{r_{i}} \SpecR(N_{i})^{(j)}\right).
	\]
	In particular, \(N\) has the \(\Rinf\)-property if and only if \(N_{i}\) has the \(\Rinf\)-property for some \(i \in \{1, \ldots, k\}\).
\end{theorem}
\begin{proof}
	The condition that \(\phi_{ij}(Z(N_{j})) = 1\) for all \(i, j \in \{1, \ldots, k\}\) with \(j \ne \sigma(i)\) follows from the proof of \cref{theo:AutomorphismGroupDirectProductFGTFNilpotentGroupsDirectyIndecomposableRationalClosure}. The rest follows from \cref{theo:AutomorphismGroupDirectProductFGTFNilpotentGroupsDirectyIndecomposableRationalClosure} itself.
\end{proof}
\subsection{Abelian factors}
Thus far, we have only considered direct products of non-abelian nilpotent groups. We now address the situation with abelian factors. We start by recalling the well-known tool for computing Reidemeister numbers on \fgtf abelian groups, see \eg \cite[\S~3]{Romankov11}.
\begin{prop}	\label{prop:ReidemeisterNumbersFreeAbelian}
	Let \(A\) be a \fgtf abelian group of rank \(r \geq 1\), so \(A \cong \Z^{r}\).
	\begin{enumerate}[1)]
		\item Let \(\phi \in \End(A)\). Then
		\[
			R(\phi) = \begin{cases}
				|\det(\Id_{A} - \phi)|	&	\mbox{if } \det(\Id_{A} - \phi) \ne 0,	\\
				\infty				&	\mbox{otherwise.}
			\end{cases}
		\]
		Here, given \(\psi \in \End(A)\), \(\det(\psi)\) is the determinant of any matrix representation of \(\psi\) w.r.t.\ a \(\Z\)-basis of \(A\).
		\item We have that
		\[
			\SpecR(A) = \begin{cases}
				\{2, \infty\}	&	\mbox{if \(r = 1\)},	\\
				\N_{0} \cup \{\infty\}	&	\mbox{otherwise.}
			\end{cases}
		\]
	\end{enumerate}
\end{prop}
\begin{theorem}
	Let \(N\) be a \fgtf nilpotent group such that \(\malclos{N}\) does not have an abelian direct factor. Let \(r \geq 1\) be an integer. Then
	\[
		\SpecR(N \times \Z^{r}) = \SpecR(N) \cdot \SpecR(\Z^{r}).
	\]
\end{theorem}
\begin{proof}
	Let \(\phi \in \Aut(N \times \Z^{r})\) and write
	\[
		\phi = \begin{pmatrix} \alpha & \beta \\ \gamma & \delta \end{pmatrix}.
	\]
	Since \(\gamma\) maps into \(\Z^{r}\), \cref{lem:isolatorCommutatorInKernelToTorsionfreeAbelian} implies that \(\sqrt[N]{\gamma_{2}(N)} \leq \ker \gamma\). As \(\malclos{N}\) does not have an abelian factor, \(Z(N) \leq \sqrt[N]{\gamma_{2}(N)}\) by \cref{prop:NoAbelianFactorRationalClosureImpliesCentreInIsolatorCommutator}. Therefore, if we look at the restricted automorphism \(\phi_{Z}\) on \(Z(N \times \Z^{r}) = Z(N) \times \Z^{r}\), we get
	\[
		\phi_{Z} = \begin{pmatrix} \alpha' & \beta \\ 0 & \delta \end{pmatrix},
	\]
	where \(\alpha'\) is the restriction of \(\alpha\) to \(Z(N)\). As \(Z(N)\) is \fgtf abelian, \(R(\phi_{Z}) = R(\alpha') R(\delta)\) by \cref{prop:ReidemeisterNumbersFreeAbelian}. Note that both \(\alpha'\) and \(\delta\) are automorphisms, since \(\phi_{Z}\) is an automorphism on a finitely generated free abelian group. By \cref{lem:endomorphsimsInducingAutomorphismOnCenterIsAutomorphismFGTFNilpotent}, \(\alpha\) is then an automorphism as well. The induced automorphism on \(\frac{N \times \Z^{r}}{Z(N \times \Z^{r})}\) has `matrix' representation
	\[
		\bar{\phi} = \begin{pmatrix} \bar{\alpha} \end{pmatrix},
	\]
	as \(\frac{N \times \Z^{r}}{Z(N \times \Z^{r})} = \frac{N}{Z(N)} \times 1\). Consequently, \(R(\bar{\phi}) = R(\bar{\alpha})\). By \cref{prop:ReidemeisterNumberEqualsProductOfCentreAndQuotientNumber}, \(R(\phi) = R(\bar{\phi})R(\phi_{Z})\). Therefore,
	\[
		R(\phi) = R(\bar{\phi})R(\phi_{Z})= R(\bar{\alpha})R(\alpha')R(\delta) = R(\alpha)R(\delta).
	\]
	Since \(\alpha\) and \(\delta\) are automorphisms, we get the inclusion \(\SpecR(N \times \Z^{r}) \subseteq \SpecR(N) \cdot \SpecR(\Z^{r})\). The other inclusion follows from \cref{prop:SpecRDirectProductCharacteristicSubgroups}.
\end{proof}
To conclude, we combine \cref{theo:SpecRDirectProductFGTFNilpotentWithAbelian} with \cref{theo:AutomorphismGroupDirectProductFGTFNilpotentGroupsDirectyIndecomposableRationalClosure}.
\begin{cor}	\label{cor:ReidemeisterSpectrumDirectProductFGTFRationallyIndecomposableNilpotentGroupsAbelianFactors}
	Let \(N_{1}, \ldots, N_{k}\) be non-isomorphic \fgtf non-abelian nilpotent groups such that \(N_{i}\) is rationally indecomposable for each \(i \in \{1, \ldots, k\}\). Let \(r_{1}, \ldots, r_{k}, r\) be positive integers. Put \(N = \Times\limits_{i = 1}^{k} N_{i}^{r_{i}}\). Then
	\[
		\SpecR(N \times \Z^{r}) = \SpecR(\Z^{r}) \cdot \prod_{i = 1}^{k} \left(\bigcup_{j = 1}^{r_{i}} \SpecR(N_{i})^{(j)}\right).
	\]
	In particular, \(N \times \Z^{r}\) has the \(\Rinf\)-property if and only if \(N_{i}\) has the \(\Rinf\)-property for some \(i \in \{1, \ldots, k\}\).
\end{cor}
\begin{proof}
	We argue that \(\malclos{N}\) has no abelian factors. Note that
	\[
		\malclos{N} = \Times_{i = 1}^{k} \left(\malclos{N}_{i}\right)^{r_{i}},
	\]
	by \cref{prop:decomposablefgtfNilpotentGroupHasDecomposableMalcevCompletion}, and that this is a decomposition into directly indecomposable groups. By \cite[Proposition~10]{BaumslagMillerOstheimer16}, such a decomposition is unique up to isomorphism and order of the factors. Since none of the \(\malclos{N}_{i}\) is abelian, \(N\) satisfies the conditions of \cref{theo:SpecRDirectProductFGTFNilpotentWithAbelian}. The result then follows after also applying \cref{theo:AutomorphismGroupDirectProductFGTFNilpotentGroupsDirectyIndecomposableRationalClosure} to \(N\).
\end{proof}
\section{Examples}
As mentioned in the introduction, it has already been proven that the conclusion from \cref{cor:ReidemeisterSpectrumDirectProductFGTFRationallyIndecomposableNilpotentGroupsAbelianFactors} holds if all \(N_{i}\) belong to either of two families: the free nilpotent groups of finite rank \cite{Tertooy19}, or the (indecomposable) \(2\)-step nilpotent groups associated to graphs \cite{DekimpeLathouwers21}. Here, we provide arguments to show that both are special cases of \cref{cor:ReidemeisterSpectrumDirectProductFGTFRationallyIndecomposableNilpotentGroupsAbelianFactors}.
\subsection{Free nilpotent groups}
For integers \(r, c \geq 2\), let \(N_{r, c}\) denote the free nilpotent group of rank \(r\) and class \(c\).
\begin{prop}	\label{prop:FreeNilpotentGroupsRationallyIndecomposable}
	Let \(r, c \geq 2\) be integers. Then \(N_{r, c}\) is rationally indecomposable.
\end{prop}
\begin{proof}
	For \(i \in \{1, \ldots, c + 1\}\), let \(\Gamma_{i} = \gamma_{i}(N_{r, c})\). By Witt's formula \cite{Witt37}, the \(\Z\)-rank of \(\Gamma_{2} / \Gamma_{3}\) equals \(\frac{r^{2} - r}{2} = \binom{r}{2}\).
	Since \(\gamma_{i}(\malclos{N}_{r, c}) = \malclos{\Gamma}_{i}\) for each \(i \in \{1, \ldots, c\}\) by \cite[Corollary~5.2.2]{CorwinGreenleaf90}, we also have that the \(\Q\)-dimension of \(\gamma_{2}(\malclos{N}_{r, c}) / \gamma_{3}(\malclos{N}_{r, c})\) equals \(\binom{r}{2}\).
	
	Under the Malcev correspondence, \(\malclos{N}_{r, c}\) corresponds to a rational nilpotent Lie algebra \(\n\). Moreover, this correspondence maps \(\gamma_{i}(\malclos{N}_{r, c})\) onto \(\gamma_{i}(\n)\). Consequently, \(\gamma_{2}(\n) / \gamma_{3}(\n)\) has \(\Q\)-dimension \(\binom{r}{2}\) as well.
	
	Now, suppose that \(\malclos{N}_{r, c}\) splits as a direct product. Then \(\n\) also splits as a direct sum of Lie subalgebras (see \eg \cite[Proposition~9]{BaumslagMillerOstheimer16}), say, \(\n = \n_{1} \oplus \n_{2}\). Then \(\n / \gamma_{2}(\n) = \frac{\n_{1}}{\gamma_{2}(\n_{1})} \oplus \frac{\n_{2}}{\gamma_{2}(\n_{2})}\). The left-hand side is an \(r\)-dimensional \(\Q\)-vector space. Hence, if \(r_{i}\) is the \(\Q\)-dimension of \(\frac{\n_{i}}{\gamma_{2}(\n_{i})}\), then \(r_{1} + r_{2} = r\). 
	
	Now, since \(\n = \n_{1} \oplus \n_{2}\) and thus in particular \([\n_{1}, \n_{2}] = 0\), \(\gamma_{2}(\n)\) is generated (as a Lie algebra) by at most
	\[
		\binom{r_{1}}{2} + \binom{r_{2}}{2}
	\]
	 elements. Consequently, the \(\Q\)-dimension of \(\gamma_{2}(\n) / \gamma_{3}(\n)\) is at most \(\binom{r_{1}}{2} + \binom{r_{2}}{2}\), which yields
	\[
		\binom{r_{1}}{2} + \binom{r_{2}}{2} \geq \binom{r}{2} = \binom{r_{1} + r_{2}}{2}.
	\]
	Working out both sides and multiplying by \(2\), we obtain
	\[
		r_{1}^{2} - r_{1} + r_{2}^{2} - r_{2} \geq (r_{1} + r_{2})^{2} - (r_{1} + r_{2}).
	\]
	Cancelling equal terms on both sides yields
	\[
		0 \geq 2r_{1} r_{2},
	\]
	which proves that \(r_{1} = 0\) or \(r_{2} = 0\). Consequently, either \(\n_{1}\) or \(\n_{2}\) is trivial, which implies that \(\n\) does not split as a non-trivial direct sum. Therefore, \(N_{r, c}^{\Q}\) is directly indecomposable.
\end{proof}
%%%% Alternatieve formulering
By this proposition, we see that \cref{cor:ReidemeisterSpectrumDirectProductFGTFRationallyIndecomposableNilpotentGroupsAbelianFactors} applies to all free nilpotent groups of finite rank, which yields the result of S.\ Tertooy \cite[Theorem~6.5.6]{Tertooy19}.

\subsection{2-step nilpotent groups associated to a graph}
Given a finite simple graph \(\Gamma(V, E)\), we can associate to it the right-angled Artin group \(A_{\Gamma}\) given by the presentation
\[
	\grppres{V}{[v, w] \text{ if \(vw \in E\)}}.
\]
Here, \(vw \in E\) means that \(v\) and \(w\) are joined by an edge. We then define \(N_{\Gamma}\) to be the \(2\)-step nilpotent quotient group \(A_{\Gamma} / \gamma_{3}(A_{\Gamma})\).
%K.\ Dekimpe and M.\ Lathouwers proved \cite[Theorem~5.11 \& Theorem~5.12]{DekimpeLathouwers21}  that the result of \cref{cor:ReidemeisterSpectrumDirectProductFGTFRationallyIndecomposableNilpotentGroupsAbelianFactors} holds for all \(N_{\Gamma}\) that are directly indecomposable\todo{preciezer formuleren}. %We provide here an argument showing that it is in fact a special case.

We can also associate a rational Lie algebra to \(\Gamma\). Let \(W\) be the \(\Q\)-vector space with basis \(V\). Let \(U\) be the subspace of \(\bigwedge^{2} W\) generated by the set \(\{v \wedge w \mid vw \not\in E\}\). We then define \(\n_{\Gamma}^{\Q}\) to be the Lie algebra with underlying vector space \(W \oplus U\) and Lie bracket defined on \(V\) by
\[
	[v, w] = \begin{cases}
		v \wedge w & \mbox{ if } vw \not\in E,	\\
		0	&	\mbox{otherwise},
	\end{cases}
\]
and \([W, U] = [U, U] = 0\). Under the Malcev correspondence, \(\n_{\Gamma}^{\Q}\) corresponds to a torsion-free \(2\)-step nilpotent group \(G_{\Gamma}^{\Q} = \{e^{x} \mid x \in W \oplus U\}\), where the group operation is given by
\begin{equation}	\label{eq:BCH2dimensions}
	e^{x} e^{y} = e^{x + y + \frac{1}{2}[x, y]_{L}}.
\end{equation}
We write \([\cdot, \cdot]_{L}\) to make clear it is the Lie bracket, not the commutator bracket. We argue that \(G_{\Gamma}^{\Q}\) is the rational completion of \(N_{\Gamma}\). If we let \(F(V)\) denote the free group on \(V\), it is readily verified that the map \(F(V) \to G_{\Gamma}^{\Q}\) sending \(v\) to \(e^{v}\) induces a group homomorphism \(F: N_{\Gamma} \to G_{\Gamma}^{\Q}\). Now, write \(V = \{v_{1}, \ldots, v_{n}\}\). Then each element \(x \in N_{\Gamma}\) can be written as 
\[
	x = \left(\prod_{i = 1}^{n} v_{i}^{k_{i}}\right)c
\]
for some \(k_{i} \in \Z\) and \(c \in \gamma_{2}(N_{\Gamma})\). From the product formula \(e^{x}e^{y} = e^{x + y + \frac{1}{2}[x, y]}\), it follows that \(F(x)\) is of the form
\[
	e^{\sum\limits_{i = 1}^{n} k_{i} v_{i} + u},
\]
for some \(u \in U\). Suppose that \(F(x) = 1\). Then \(\sum\limits_{i = 1}^{n} k_{i} v_{i} + u = 0\). As \(\sum\limits_{i = 1}^{n}k_{i}v_{i} \in W\) and \(u \in U\), we must have \(\sum\limits_{i = 1}^{n} k_{i}v_{i} = 0\) and \(u = 0\). Therefore, \(k_{i} = 0\) for all \(i \in \{1, \ldots, n\}\), as \(V\) is a basis of \(W\). Hence, \(x = c \in \gamma_{2}(N_{\Gamma})\). As \(N_{\Gamma}\) is \(2\)-step nilpotent, we can write
\[
	c = \prod_{\substack{1 \leq i < j \leq n \\ v_{i}v_{j} \not\in E}}[v_{i}, v_{j}]^{l_{ij}}
\]
for some \(l_{ij} \in \Z\). Using \eqref{eq:BCH2dimensions}, it is readily verified that \([e^{y}, e^{z}] = e^{[y, z]_{L}}\) for all \(y, z \in L\). Therefore, \(F(x) = F(c) = e^{z}\), where
\[
	z = \sum\limits_{\substack{1 \leq i < j \leq n \\ v_{i}v_{j} \not\in E}}l_{ij}[v_{i}, v_{j}]_{L}.
\]
Since \(F(x) = 1\), we must have \(z = 0\). Consequently, as \(\{[v_{i}, v_{j}]_{L} \mid 1 \leq i < j \leq n, v_{i}v_{j} \not\in E\}\) is a basis of \(U\), we must have that \(l_{ij} = 0\) for all \(i, j\), which implies that \(c = 1\) as well. We conclude that \(F\) is indeed injective, so \(N_{\Gamma}\) embeds in \(G_{\Gamma}^{\Q}\).

We now verify the other two properties for \(G_{\Gamma}^{\Q}\) to be the rational completion of \(N_{\Gamma} \cong \im F\). Let \(e^{w + u} \in \im F\) and \(k \in \N_{0}\). Then
\[
	e^{w + u} = (e^{\frac{1}{k}(w + u)})^{k}.
\]
Hence, all elements of \(\im F\) have roots in \(G_{\Gamma}^{\Q}\), which are unique as \(G_{\Gamma}^{\Q}\) is torsion-free and nilpotent.

Next, as \([e^{y}, e^{z}] = e^{[y, z]_{L}}\) for all \(y, z \in L\), \(e^{k[v, w]_{L}} \in \im F\) for all \(k \in \Z\) and \(v, w \in V\). Combining this with the fact that \(e^{[x, y]_{L}}e^{[z, w]_{L}} = e^{[x, y]_{L} + [z, w]_{L}}\), we get that
\begin{equation}	\label{eq:integralExponentialCommutatorInImageF}
	\forall u = \sum_{v, w \in V} \mu_{v, w}[v, w]_{L}: (\forall v, w \in V: \mu_{v, w} \in \Z) \implies e^{u} \in \im F.
%	e^{u} \in \im F \text{ for all \(u \in U\) of the form } u = \sum_{v, w \in V} \mu_{v, w} [v, w]_{L} \text{ with \(\mu_{v, w} \in \Z\).}
\end{equation}
Furthermore, if we take \(\lambda_{i} \in \Z\) for all \(i \in \{1, \ldots, n\}\), then a straightforward calculation yields that
\begin{equation}	\label{eq:evenIntegralExponentialVerticesInImageF}
	e^{\sum_{i = 1}^{n} 2\lambda_{i}v_{i}} = \prod_{i = 1}^{n}(e^{v_{i}})^{2\lambda_{i}} \cdot \prod_{1 \leq i < j \leq n}e^{-2\lambda_{i} \lambda_{j}[v_{i}, v_{j}]} \in \im F.
\end{equation}
Finally, let \(x + u \in W \oplus U\). Write
\[
	x + u = \sum_{i = 1}^{n} \lambda_{i} v_{i} + \sum_{\substack{1 \leq i < j \leq n \\ v_{i}v_{j} \not\in E}} \mu_{i, j} [v_{i}, v_{j}]_{L}
\]
for some \(\lambda_{i}, \mu_{i, j} \in \Q\). Then there is an \(m \in \N_{0}\) such that \(m \lambda_{i}, m \mu_{i, j} \in 2\Z\) for all \(1 \leq i < j \leq n\). Then
\[
	(e^{x + u})^{m} = e^{mx + mu} \in \im F,
\]
by using \eqref{eq:integralExponentialCommutatorInImageF} and \eqref{eq:evenIntegralExponentialVerticesInImageF}. Hence, the third property is satisfied, which shows that \(G_{\Gamma}^{\Q}\) is indeed the rational completion of \(N_{\Gamma}\).
%Let \(\f_{2}^{\Q}(V)\) be the free \(2\)-step nilpotent rational Lie algebra on \(S\). Let \(I\) be the ideal in \(\f_{2}^{\Q}(V)\) generated by \(\{[v, w] \mid vw \in E\}\). Then we put \(\n_{\Gamma}^{\Q} = \f_{2}^{\Q}(V) / I\). This Lie

We recall the definition of the simplicial join of two graphs. Given two graphs \(\Gamma_{i}(V_{i}, E_{i})\), \(i \in \{1, 2\}\), the \emph{simplicial join} \(\Gamma_{1} * \Gamma_{2}\) is the graph \(\Gamma(V, E)\) where \(V = V_{1} \sqcup V_{2}\) and
\[
	E = E_{1} \sqcup E_{2} \sqcup \{vw \mid v \in V_{1}, w \in V_{2}\}.
\]

We say that a graph \(\Gamma\) \emph{splits as a simplicial join} if \(\Gamma = \Gamma_{1} * \Gamma_{2}\) for some non-trivial graphs \(\Gamma_{1}\) and \(\Gamma_{2}\).

\begin{prop}
	Let \(\Gamma\) be a finite simple graph. Then \(N_{\Gamma}\) is rationally indecomposable if and only if \(\Gamma\) does not split as a simplicial join.
\end{prop}
\begin{proof}
	Suppose that \(\Gamma = \Gamma_{1} * \Gamma_{2}\) is a non-trivial simplicial join. Then \(A_{\Gamma} \cong A_{\Gamma_{1}} \times A_{\Gamma_{2}}\). Consequently, \(N_{\Gamma}\) splits as a non-trivial direct product, which by \cref{prop:decomposablefgtfNilpotentGroupHasDecomposableMalcevCompletion} yields a decomposition of \(\malclos{N}_{\Gamma} = \malclos{G}_{\Gamma}\).
	
	Conversely, suppose that \(N_{\Gamma}\) is rationally decomposable. Similarly as in \cref{prop:FreeNilpotentGroupsRationallyIndecomposable}, the Lie algebra \(\n_{\Gamma}^{\Q}\) then admits a non-trivial decomposition \(\n_{\Gamma}^{\Q} = \n_{1} \oplus \n_{2}\). For \(i \in \{1, 2\}\), let \(\beta_{i}\) be a \(\Q\)-vector space basis of \(\n_{i}\). Write \(\beta_{i} = \{w_{i, 1} + u_{i, 1}, \ldots, w_{i, k_{i}} + u_{i, k_{i}}\}\) for some \(w_{i, j} \in W, u_{i, j} \in U\). Write \(\{v^{*}\}_{v \in V}\) for the dual basis of \(W\) associated to \(V\). Define, for \(i = 1, 2\),
	\[
		V_{i} := \{v \in V \mid \exists j \in \{1, \ldots, k_{i}\}: v^{*}(w_{i, j}) \ne 0\}.
	\]
	Since \(\Span_{\Q}(\beta_{1} \cup \beta_{2}) = W \oplus U\), we have that
	\[
		W = \Span_{\Q}(\{w_{i, j} \mid i \in \{1, 2\}, j \in \{1, \ldots, k_{i}\}\}).
	\]
	Consequently, for each \(v \in V\), there exists a \(w_{i, j}\) such that \(v^{*}(w_{i, j}) \ne 0\). Hence, \(V_{1} \cup V_{2} = V\). Next, we argue that neither \(V_{1}\) nor \(V_{2}\) is empty. Suppose, for the sake of contradiction, that \(V_{1}\) is empty. Then \(W = \Span_{\Q}(\{w_{2, j} \mid j \in \{1, \ldots, k_{2}\}\}\), which implies that, for each \(v \in V\), there exist a \(u \in U\) such that \(v + u \in \n_{2}\). Since \(\n_{2}\) is closed under the Lie bracket, this implies that
	\[
		\{[v, v'] \mid v, v' \in V\} \subseteq \n_{2}.
	\]
	From this, it follows that \(U \subseteq \n_{2}\). As there exist, for each \(v \in V\), a \(u \in U\) such that \(v + u \in \n_{2}\), \(V\) lies in \(\n_{2}\) as well. Consequently, \(W = \Span_{\Q}(V)\) also lies in \(\n_{2}\), which implies that \(\n_{2} = W + U = \n\). This contradicts the non-triviality of \(\n_{1}\). We conclude that \(V_{1}\) is non-empty, and a similar argument holds for \(V_{2}\).
	
	As \([\n_{1}, \n_{2}] = 0\), we have that \([w_{1, i}, w_{2, j}] = 0\) for all \(i \in \{1, \ldots, k_{1}\}\) and \(j \in \{1, \ldots, k_{2}\}\). Consequently, every \(v \in V_{1}\) is connected to every \(w \in V_{2}\) in \(\Gamma\) (if \(v \ne w\)). If \(V_{1} = V_{2} = V\), then \(\Gamma\) is a complete graph, which clearly splits as a simplicial join (note that \(\size{V} \geq 2\) in that case, as \(G^{\Q}_{\Gamma} \cong \Q^{\size{V}}\) is decomposable by assumption).
	
	So, assume either \(V_{1}\) or \(V_{2}\) is a strict subset of \(V\). Then either \(\{V_{1} \setminus V_{2}, V_{2}\}\) or \(\{V_{1} ,V_{2} \setminus V_{1}\}\) forms a partition \(\{W_{1}, W_{2}\}\) of \(V\) such that \(\Gamma = \Gamma(W_{1}) * \Gamma(W_{2})\), which proves that \(\Gamma\) splits as a non-trivial simplicial join.
\end{proof}
By this proposition, we see that \cref{cor:ReidemeisterSpectrumDirectProductFGTFRationallyIndecomposableNilpotentGroupsAbelianFactors} applies to all groups \(N_{\Gamma}\) where \(\Gamma\) does not split as a simplicial join, which yields the result of K.\ Dekimpe and M.\ Lathouwers \cite[\S5]{DekimpeLathouwers21}.

Finally, by \cref{prop:DirectlyIndecomposableMalcevCompletionImpliesDirectlyIndecomposable}, all finite index subgroups of free nilpotent groups or directly indecomposable \(2\)-step nilpotent quotients of RAAGs satisfy the conditions of \cref{cor:ReidemeisterSpectrumDirectProductFGTFRationallyIndecomposableNilpotentGroupsAbelianFactors} as well.

\section*{Acknowledgements}
The author thanks Karel Dekimpe for the useful remarks and suggestions, and Maarten Lathouwers and Thomas Witdouck for the helpful discussions.

\printbibliography

\end{document}